\documentclass[11pt]{article}
\usepackage{amsfonts,latexsym,amsmath,amscd,geometry}
\usepackage{graphicx}
\usepackage{amssymb}

\usepackage{amssymb,mathrsfs,mathtools,graphicx,enumerate,color,colortbl}
\usepackage{hyperref}
\usepackage{comment}

\usepackage{latexsym}
\usepackage{xcolor}
\usepackage{indentfirst}

\newcommand \nc{\newcommand}
\newtheorem{theorem}{Theorem}[section]
\newtheorem{lemma}[theorem]{Lemma}

\newtheorem{remark}[theorem]{Remark}

\nc{\ba}{\begin{array}}\nc{\ea}{\end{array}}
\nc{\be}{\begin{eqnarray}}\nc{\ee}{\end{eqnarray}}
\nc{\beq}{\begin{equation}}\nc{\eeq}{\end{equation}}
\nc{\bex}{\begin{eqnarray*}}\nc{\eex}{\end{eqnarray*}}
\nc{\btm}{\begin{theorem}} \nc{\etm}{\end{theorem}}
\nc{\blm}{\begin{lemma}} \nc{\elm}{\end{lemma}}
\nc{\R}{\mathbb{R}}  

\newcommand \qed {\hfill $\Box$}

\newcommand{\Xcal}{\mathcal{X}}

\renewcommand{\hbar}{\bar{h}}

\newcommand{\lbar}{\bar{\lambda}}

\newcommand{\util}{\tilde{u}}

\renewcommand{\d}{\delta}

\renewcommand{\l}{\lambda}

\newcommand{\x}{\xi}

\newcommand{\s}{\sigma}

\newcommand{\e}{\varepsilon}

\newcommand{\RR}{{\mathbb R}}

\newcommand{\abs}[1]{\left|#1\right|}

\newcommand{\norm}[1]{\left\|#1\right\|}

\newcommand{\case}[1]{\noindent{\it Case #1:} }

\definecolor{black}{rgb}{0.0, 0.0, 0.0}
\definecolor{red}{rgb}{1.0, 0.5, 0.5}

\newcommand{\uubar}{\underline{u}}
\newcommand{\lubar}{\underline{\lambda}}

\begin{document}

\title{\(L^2\)-contraction of Shock Waves for KdV--Burgers Equation}

\author{Geng Chen
\footnote{Department of Mathematics, University of Kansas, Lawrence, KS 66045, U.S.A. Email: gengchen@ku.edu},
\quad
Namhyun Eun
\footnote{School of Mathematics, Korea Institute for Advanced Science (KIAS), Seoul 02455, Republic of Korea. Email: namhyuneun@kias.re.kr},
\quad
Moon-Jin Kang
\footnote{Department of Mathematical Sciences, Korea Advanced Institute of Science and Technology, Daejeon 34141, Republic of Korea. Email: moonjinkang@kaist.ac.kr},
\quad
Yannan Shen
\footnote{Department of Mathematics, University of Kansas, Lawrence, KS 66045, U.S.A. Email: yshen@ku.edu}
}

\maketitle

\begin{abstract}
The KdV--Burgers equation is a canonical model describing the interplay between nonlinearity, viscosity and dispersion, and it admits viscous-dispersive shocks as traveling wave solutions.
%When viscosity dominates, these shocks are monotone, whereas when dispersion dominates, they become non-monotone and exhibit infinite oscillations.
In this paper, we establish an \(L^2\)-contraction property for viscous-dispersive shocks under arbitrarily large perturbations, up to a time-dependent shift.
This yields time-asymptotic stability and uniform estimates with respect to the strengths of viscosity and dispersion.
We present the proof for the monotone shocks, and introduce the companion work in \cite{CEKS} on the stability and structural properties of oscillatory shocks.
%In addition, we identify structural properties of the monotone profiles, namely, the exponential decay of the profiles and their derivatives.

%\subjclass{35B35; 35Q53; 35L67; 76L05}

%\keywords{Korteweg--de Vries--Burgers equation; Viscous-dispersive shock, Contraction, Uniform stability, Zero viscosity-dispersion limit}
\end{abstract}

\section{Introduction}
\setcounter{equation}{0}
This paper is devoted to the analysis of the Korteweg--de Vries--Burgers (KdVB) equation:
\begin{equation} \label{burgers}
u_t + f(u)_x = \e u_{xx}-\d u_{xxx}, \qquad u(0,x)=u_0(x),
\end{equation}
where \(u\colon \RR^+ \times \RR \to \RR\) is the unknown function.
The flux is given by \(f(u)=u^2/2\) and the constants \(\e,\d>0\) respectively denote the viscosity and dispersion coefficients.
The dispersive model without viscosity was first introduced by Boussinesq \cite{B1872} and later by Korteweg--de Vries \cite{KdV1895}.
The form given in \eqref{burgers} was subsequently derived by Su--Gardner \cite{SG-JMP}.
Note that this equation is among the most fundamental models describing the interplay of nonlinearity (the flux), dissipation (viscosity), and dispersion.
It is therefore natural that this equation appears in a wide range of physical applications, including shallow water waves \cite{B1872}, plasma physics \cite{Biskamp}, traffic flow \cite{Lighthill}, optical fibers \cite{Xu}, nematic liquid crystals \cite{Smyth}, quantum fluids \cite{HoeferAblowitz} and nonlinear dynamical lattices \cite{Biondini}.
The KdVB equation \eqref{burgers} can be regarded as a canonical model of the Navier--Stokes--Korteweg system, which describes compressible fluids with viscosity and capillarity \cite{DunnSerrin}.
Moreover, the dynamics of shock waves for convex and nonconvex fluxes have been extensively studied; see \cite{Heofer} and the references therein.

\vspace{2mm}
A key feature of the KdVB equation is the existence of viscous-dispersive shocks.
In 1985, Bona--Schonbek \cite{bona} showed this equation admits traveling wave solutions, commonly referred to as viscous-dispersive shocks.
To be more precise, it has solutions of the form \(\util(\x)=\util(x-\s t)\), which satisfies the following ordinary differential equation:
\begin{equation} \label{visS}
-\s \util' + f(\util)' =\e\util'' -\d\util'''
\end{equation}
together with the asymptotic conditions
\[
\util(\x) \to u_\pm, \quad
\util'(\x) \to 0, \quad
\text{as } \x \to \pm \infty,
\]
where \(u_->u_+\) and the traveling speed \(\s\) is determined by the Rankine--Hugoniot condition
\begin{equation} \label{sigma}
\s = \frac{f(u_+)-f(u_-)}{u_+-u_-} = \frac{u_-+u_+}{2}.
\end{equation}
They also established that the qualitative structure of the associated shocks depends on the balance between viscosity and dispersion.
When viscosity dominates dispersion, namely, \(\e^2 \ge 2\d(u_--u_+)\), the associated shock wave profiles are monotone.
In contrast, in the dispersion-dominated regime \(\e^2 < 2\d(u_--u_+)\), the profiles become non-monotone and exhibit infinitely many oscillations.

\vspace{2mm}
The stability analysis of traveling wave solutions has been a topic of considerable interest in the literature.
Matsumura--Nishihara \cite{MN-1985} proved the stability of viscous shocks for the compressible Naiver--Stokes (NS) equations.
Around the same time, Goodman \cite{Goodman} derived an analogous stability result for the general system endowed with artificial diffusions.
%These results were improved by Liu \cite{Liu} and by Szepessy--Xin \cite{SzepessyXin}.
Then, employing a distinct analytical approach, Mascia--Zumbrun \cite{MasciaZumbrun} established the spectral stability of viscous shock profiles for the compressible NS equations.
Recently, Kang--Vasseur \cite{acon_i_2} developed the method of a-contraction with shifts, a highly nonlinear energy method designed to obtain the stability of viscous shocks up to time-dependent shifts.
One notable strength of this approach is its ability to handle arbitrarily large perturbations, which yield uniform stability for inviscid limits.
Based on this method, the stability of viscous shocks under arbitrarily large perturbations has been established for several equations, including the viscous Burgers equation \cite{KangVasseur}, the barotropic NS system \cite{EEKO-ISO,acon_i_2,KV}, the Brenner--Navier--Stokes--Fourier system \cite{EEK-BNSF}.
Moreover, the stability of viscous-dispersive shocks of the Naiver--Stokes--Korteweg system was shown in \cite{EKK-NSK}.

\vspace{2mm}
However, despite the physical importance and mathematical significance of the KdVB equation, the stability of its viscous-dispersive shocks remains poorly understood.
In 1985, Pego \cite{Pego}, building on the methods of Goodman \cite{Goodman} and Matsumura--Nishihara \cite{MN-1985}, established the time-asymptotic stability of monotone shock profiles under small perturbations.
This result was subsequently extended, via perturbative arguments, to non-monotone shock profiles \cite{Khodja,Naumkin}. 
It should be noted, however, that their analysis applies only to oscillatory shocks sufficiently close to the monotone regime.
We also refer the readers to \cite{Zumbrun2,Zumbrun1} for further results concerning the stability of traveling wave solutions.
More recently, in 2025, Barker--Bronski--Hur--Yang \cite{Hur} proposed a novel approach that connects shock stability to the spectral properties of an associated Schr\"odinger operator, providing an analytic proof for the monotone shocks and employing computer-assisted analysis to investigate non-monotone shocks.
To the best of our knowledge, apart from \cite{Hur}, no results have been available addressing the stability of genuinely oscillatory shocks for more than four decades since their existence was established \cite{bona}.

\vspace{2mm}
The objective of this paper is to remove the smallness condition on perturbations imposed in \cite{Pego}, i.e., we focus on the monotone shock regime and show the \(L^2\) stability of the corresponding shock profiles.
We also aim to announce a companion work \cite{CEKS} to the present paper, where the authors established an \(L^2\)-contraction property for oscillatory shocks under arbitrarily large perturbations.
In both cases, we adopt one key component of the a-contraction method with shifts, namely time-dependent shifts (or temporal modulations).
The absence of any size restriction on perturbations ensures uniform estimates with respect to the coefficients \(\e\) and \(\d\), which in turn allow us to justify zero viscosity-dispersion limits.

\subsection{Main Result}
In the following, we state the main theorems.
To describe the class of perturbations, we introduce the following function space:
\[
\Xcal_T \coloneqq
\{u\colon\RR^+\times\RR\to\RR \mid u-\uubar \in C([0,T];H^1(\RR)) \cap L^2(0,T;H^2(\RR))\}
\]
where \(\uubar\) is a smooth monotone function which satisfies
\[
\uubar(x) \coloneqq
\begin{cases}
u_-, &x\le-1,\\
u_+, &x\ge+1.
\end{cases}
\]
Then we present the main theorem on an \(L^2\)-contraction for monotone viscous-dispersive shocks.
This holds for any large perturbation, from which the time-asymptotic and uniform stability follow.
\begin{theorem} \label{thm:main}
Let \(\e,\d>0\) be constants and let \(u_\pm\) be given states.
Assume that \(u_->u_+\) and 
\begin{equation} \label{cond1}
0 < \frac{\d(u_--u_+)}{2\e^2} \le \frac{1}{4}.
\end{equation}
Let \(\util\) be the associated viscous-dispersive shock of \eqref{burgers} which monotonically connects \(u_-\) and \(u_+\).
Let \(u_0\) be given initial data with \(\norm{u_0-\util}_{H^1(\RR)}<+\infty\) and let \(u\in\Xcal_T\) denote the solution of \eqref{burgers} subject to the initial data \(u_0\).
Then, the following \(L^2\)-contraction holds:
%\begin{equation} \label{L2-stable}
%\begin{aligned}
%&\int_\RR \big(u(t,x)-\util^{-X}(x-\s t)\big)^2 dx
%+\frac{(u_--u_+)}{M}\int_0^T |\dot{X}(t)|^2 dt\\
%&\qquad +\Big(2\sqrt{2}-\frac{5}{2}\Big)\e \int_0^T \int_\RR \big((u(t,x)-\util^{-X}(x-\s t))_x\big)^2 dx dt
%\le \int_\RR \big(u_0(x)-\util(x)\big)^2 dx.
%\end{aligned}
%\end{equation}
\begin{equation} \label{L2-stable}
\begin{aligned}
&\int_\RR \big(u(t,x)-\util(x-\s t-X(t))\big)^2 dx+ C^*\e \int_0^T \int_\RR \big((u(t,x)-\util(x-\s t-X(t)))_x\big)^2 dx dt\\
&\hspace{110mm}
\le \int_\RR \big(u_0(x)-\util(x)\big)^2 dx,
\end{aligned}
\end{equation}
where \(C^* \coloneqq 2\sqrt{2}-\frac{5}{2}\) and the Lipschitz shift function \(X(t)\) is determined by \(X(0)=0\) and
%\begin{equation} \label{X-def}
%\dot{X}(t)= \frac{2M}{(u_--u_+)} \int_\RR \big(u(t,x+X(t))-\util(x-\s t)\big) \, \util'(x-\s t)dx,
%\end{equation}
\begin{equation} \label{X-def}
\dot{X}(t)= \frac{1}{(u_--u_+)} \int_\RR \big(u(t,x+X(t))-\util(x-\s t)\big) \, \util'(x-\s t)dx.
\end{equation}
Moreover, for all \(2<p\le\infty\), the following holds:
\begin{equation} \label{timeasy}
\norm{u(t,\cdot)-\util(\cdot-X(t))}_{L^p(\RR)} \to 0 \quad \text{as } t \to \infty,
\end{equation}
which shows the time-asymptotic stability.
Lastly, the shift function \(X(t)\) satisfies
\begin{equation} \label{shiftasy}
|\dot{X}(t)| \to 0 \quad \text{as } t \to \infty.
\end{equation}
\end{theorem}

In the companion paper \cite{CEKS}, the same contraction property was established for the non-monotone regime \(\frac{1}{4} < \frac{\d(u_--u_+)}{2\e^2} < \frac{1}{2}\).
Despite the identical choice of the shift function, the oscillatory nature significantly complicates the analysis.
In that work, the authors developed an inductive argument to effectively exploit the shift function.
Complimenting \cite{CEKS}, the present paper covers the remaining interval, which corresponds to the monotone regime.
The following remark is in order regarding the shift function.

\begin{remark}
(1) The shift employed here is motivated by previous studies on traveling wave solutions of the viscous Burgers equation \cite{KangVasseur} and the KdVB equation \cite{Hur}.
As discussed in \cite{Hur}, it admits a natural interpretation as a gradient flow minimizing the \(L^2\)-distance to the shocks.\\
(2) In the companion work \cite{CEKS}, the global existence of solutions to \eqref{burgers} was shown.
To be precise, for initial data \(u_0\) with \(\norm{u_0-\uubar}_{H^1(\RR)}<+\infty\), there exists a unique global solution in \(\Xcal_T\).
The shift is well-defined by the Cauchy--Lipschitz theorem, provided \(u\in\Xcal_T\) (see \cite[Remark 1.1]{KangVasseur}).\\
%In \cite{Hur}, the global existence is assumed for solutions evolving from initial data in \(H^2\).\\
(3) As the shift function \(X(t)\) satisfies the time-asymptotic behavior \eqref{shiftasy}, the following holds: 
\[
\frac{X(t)}{t} \to 0 \quad \text{as } t\to\infty.
\]
Therefore, \(X(t)\) grows at most sub-linearly, so that the leading-order propagation of the shifted shock remains governed by the original speed \(\s\).
\end{remark}

\begin{remark}
Theorem \ref{thm:main} does not impose any size restriction on perturbations, i.e., as long as the initial data \(u_0\) satisfies \(\norm{u_0}_{H^1(\RR)}<\infty\), the contraction holds.
This improves the result in \cite{Pego} and provides an alternative energy-method proof of \cite{Hur} for the monotone shocks.
Furthermore, as shown in \cite{CEKS}, the absence of size restrictions yields uniformity-in-the strengths of viscosity and dispersion.
Thus, we recover \cite[Theorem 1.2]{CEKS}, which shows the zero viscosity-dispersion limits on which the associated Riemann shocks are unique and orbitally stable.
However, as both the statement and the proof are essentially the same, we omit them for brevity.
\end{remark}

In the companion paper \cite{CEKS}, the structural properties of the shocks were identified.
In particular, the authors demonstrated the decay rate at which the local extrema converge to the left end state towards the left far field.
By contrast, the monotone viscous-dispersive shock profiles admit no local extrema, but instead exhibit exponential decay, which is manifested in the following theorem.

\begin{theorem} \label{thm:exp}
Let \(\e,\d>0\) be constants, and \(u_\pm\) be given states.
Assume that \(u_->u_+\) and \eqref{cond1}.
Let \(\util\) be the associated monotone viscous-dispersive shock of \eqref{burgers} such that \(\util(0)=\frac{u_-+u_+}{2}\).
Then,
\[
\frac{(u_--u_+)}{2} e^{-\frac{\lbar(u_--u_+)\abs{x}}{\e}}
\le u_--\util(x)
\le \frac{(u_--u_+)}{2} e^{-\frac{\lubar(u_--u_+)\abs{x}}{2\e}}, \qquad \forall x\le0,
\]
\[
\frac{(u_--u_+)}{2} e^{-\frac{\lbar(u_--u_+)\abs{x}}{\e}}
\le \util(x)-u_+
\le \frac{(u_--u_+)}{2} e^{-\frac{\lubar(u_--u_+)\abs{x}}{2\e}}, \qquad \forall x\ge0,
\]
\[
\frac{\lubar(u_--u_+)^2}{4\e} e^{-\frac{\lbar(u_--u_+)\abs{x}}{\e}}
\le -\util'(x)
\le \frac{\lbar(u_--u_+)^2}{2\e} e^{-\frac{\lubar(u_--u_+)\abs{x}}{2\e}}, \qquad \forall x\in\RR,
\]
where \(\lubar=\sqrt{2}-1\) and \(\lbar=1\).
\end{theorem}
\begin{remark}
All monotone shocks, that is, traveling wave solutions to \eqref{burgers} satisfying \eqref{cond1}, exhibit the exponential decay described in Theorem \ref{thm:exp}.
However, the proofs of Lemma \ref{lem:key} and Theorem \ref{thm:exp} show that the decay rates, i.e., \(\lubar\) and \(\lbar\), depend on the parameters \((\e,\d,u_\pm)\).
In other words, sharper exponential tails can be obtained under a more restrictive regime than \eqref{cond1}.
\end{remark}

Throughout the rest of the paper, we henceforth assume that \(\d>0\).
If \(\d<0\), the change of variables \(x\to -x, u\to -u, \d\to -\d\) yields the same result.
Moreover, due to the Galilean invariant transformation \(u(t,x-\s t)+\s \to u(t,x)\), we may also assume \(\s=0\) without loss of generality, i.e., the far field states are given by
\begin{equation} \label{stny}
u_\pm = \lim_{x\to\pm\infty} \util(x) = \mp s
\end{equation}
for some \(s>0\), in accordance with the Rankine--Hugoniot condition \eqref{sigma}.

\section{Preliminaries}
\setcounter{equation}{0}
%\subsection{Evolution of \(L^2\)-Energy}
%To prove the contraction property \eqref{L2-stable}, we demonstrate that the %time-derivative of the \(L^2\)-norm is non-positive.
%The following lemma characterizes the quadratic structure of 
%\[
%\frac{d}{dt}\frac{1}{2}\int_\RR (u^X-\util)^2 dx = \frac{d}{dt}\frac{1}%{2}\int_\RR (u-\util^{-X})^2 dx
%\]
We first observe the time-derivative of the \(L^2\)-norm.
We denote \(u^X(t,x) = u(t,x+X(t))\).
\begin{lemma} \label{lem:RHS}
Let \(\util\) denote the viscous-dispersive shock with \eqref{visS} and \eqref{stny}.
Then, for any solution \(u\in\Xcal_T\) to \eqref{burgers} and any Lipschitz function \(X\colon [0,T]\to\RR\), the following holds.
\begin{equation} \label{RHS}
\frac{d}{dt}\frac{1}{2}\int_\RR (u^X-\util)^2 dx
=\dot{X}(t) \int_\RR (u^X-\util)\util'dx
-\frac{1}{2}\int_\RR (u^X-\util)^2 \util' dx
-\e\int_\RR ((u^X-\util)_x)^2 dx.
\end{equation}
\end{lemma}

\begin{pf}
The proof does not depend on the critical parameter \(\frac{\d(u_--u_+)}{2\e^2}\) in \eqref{cond1}, and thus is given in the companion paper \cite[Appendix B]{CEKS}.
We also refer the readers to \cite{Hur,KangVasseur}.
\end{pf}

%\subsection{Poincar\'e Type Inequality}
\vspace{2mm}
The proof of the contraction property relies on the following Poincar\'e-type inequality, which enables us to exploit the shift function and the dissipation term.
\begin{lemma}\label{KV_lemma} \cite[Lemma 2.9]{acon_i_2}
For any function \(f\colon[a,b]\to\RR\) satisfying \(\int_a^b (y-a)(b-y)\abs{f'}^2 dy<\infty\),
\begin{equation} \label{KV}
\int_a^b f^2dy
\le \frac{1}{2}\int_a^b(y-a)(b-y)\abs{f'}^2 dy
+\frac{1}{b-a} \Big(\int_a^b f dy\Big)^2.
\end{equation}
\end{lemma}
%Note that the coefficient \(\frac{1}{2}\) is independent of the length of the domain \([a,b]\); see \cite[Lemma 2.9]{acon_i_2}.
%The first term on the right-hand side arises from the dissipation, while the second term arises from the shift term.
Note that this lemma has proven to be a powerful tool in the study of the stability of viscous shocks \cite{acon_i_2,KV}, and composite wave patterns \cite{KVW-ADV,KVW-ARMA}.

%\subsection{Estimates for the Shock Derivatives}
\vspace{2mm}
Now we derive estimates for the derivatives of the shock profiles, which will be useful later.
More precisely, in the proof of Theorem \ref{thm:main}, these estimates will be used to transfer the dissipation into the first term on the right-hand side of \eqref{KV}.
Moreover, these together with a standard Gr\"onwall argument yield the desired exponential decays.

\begin{lemma} \label{lem:key}
Assume that \(0 < 2\d(u_--u_+) = 4\d s \le \e^2\).
Then, the following holds:
\begin{equation} \label{key:ineq}
\lubar (s-\util(x))(\util(x)+s)
\le -\e \util'(x)
\le \lbar (s-\util(x))(\util(x)+s), \qquad \forall x\in\RR,
\end{equation}
where \(\lubar = \sqrt{2}-1\) and \(\lbar = 1\).
\end{lemma}

%\begin{remark}
%In the absence of dispersion, i.e., when \(\d=0\), the shock derivative satisfies 
%\[
%-\e \util' = \frac{1}{2}(s-\util)(s+\util),
%\]
%which is no longer available in the present setting (see \eqref{t0}).
%Nevertheless, Lemma \ref{lem:key} shows that the shock derivatives remain comparable with \((s-\util)(s+\util)\), from which the exponential decay properties in Theorem \ref{thm:exp} follow.
%Further, the lower bound of \((-\e\util')\) in \eqref{key:ineq} will play a crucial role in the proof of Theorem \ref{thm:main}.
%\end{remark}

\begin{pf}
The proof relies on the contradiction argument developed in the companion paper \cite{CEKS}.
To this end, we first introduce a parameterization in terms of \(a=\util\) and define the following functions 
\[
h(a)\coloneqq \util' = \util_x\quad\hbox{and}\quad
p_\l(a) = \frac{\l}{\e} (a-s)(a+s) \ \ \hbox{for any } \l>0.
\]
The function \(h\) is well defined as the shock is monotone.
Under this parametrization, the desired inequality \eqref{key:ineq} can be rewritten into the following form:
\begin{equation} \label{key:ineq2}
p_{\lbar}(a) \le h(a) \le p_{\lubar}(a), \qquad \forall a\in(-s,s).
\end{equation}
To proceed, we examine the above functions at the endpoints as follows: for any \(\l>0\),
\begin{equation} \label{h-p}
h(-s)=p_\l(-s)=0, \qquad
h(s)=p_\l(s)=0.
\end{equation}
Now we also investigate their derivatives \(h'\) and \(p'\).
We integrate \eqref{visS} over \((\pm\infty,x)\) to find that
\begin{equation} \label{t0}
\e\util_x -\d\util_{xx} = \frac{1}{2}(\util^2-s^2).
\end{equation}
Dividing it by \(\util_x\), we obtain
\begin{equation} \label{h'}
h'(a)
= \frac{d\util_x}{d\util}
= \frac{\util_{xx}}{\util_x}
= \frac{1}{\d} \Big[\e - \frac{1}{2h(a)}(a-s)(a+s)\Big].
\end{equation}
Thus, we have 
\[
\lim_{a\to-s}h'(a)
=\lim_{a\to-s}\frac{1}{\d} \Big[\e - \frac{(a-s)}{2}\frac{(a+s)}{h(a)}\Big], \qquad
\lim_{a\to s}h'(a)
=\lim_{a\to s}\frac{1}{\d} \Big[\e - \frac{(a+s)}{2}\frac{(a-s)}{h(a)}\Big].
\]
It follows that 
\[
h'(-s) = \frac{1}{\d} \Big[\e - \frac{(a-s)}{2h'(-s)}\Big], \qquad
h'(s) = \frac{1}{\d} \Big[\e - \frac{(a+s)}{2h'(s)}\Big].
\]
This implies that
\[
\d (h'(-s))^2 - \e (h'(-s)) -s =0, \qquad
\d (h'(s))^2 - \e (h'(s)) +s =0.
\]
Then, since \(h'(-s)<0\) and \eqref{h'} indicates that \(h'(a)<\frac{\e}{\d}\) on \((-s,s)\), we have 
\[
h'(-s) = \frac{\e-\sqrt{\e^2+4\d s}}{2\d}, \qquad
h'(s) = \frac{\e-\sqrt{\e^2-4\d s}}{2\d}.
\]
Recall that \(\e^2 \ge 4\d s\).
We now distinguish two cases: \(\l<\lubar = \sqrt{2}-1\) and \(\l>\lbar = 1\).

\vspace{2mm}
\case{1} We consider the first case when \(\l<\lubar = \sqrt{2}-1\).\\
Since we have
\[
\l < \sqrt{2}-1 = \frac{1}{1+\sqrt{2}}
< \frac{\e}{\e+\sqrt{\e^2+4\d s}},
\]
it holds that
\begin{equation} \label{hp-ns}
- p_\l'(-s) = 2\frac{\l}{\e}s < \frac{2s}{\e+\sqrt{\e^2+4\d s}}
= -\frac{\e-\sqrt{\e^2+4\d s}}{2\d} = -h'(-s).
\end{equation}
Moreover, since we have 
\[
\l < \frac{1}{2}
< \frac{\e}{\e+\sqrt{\e^2-4\d s}},
\]
it follows that
\begin{equation} \label{hp-s}
p_\l'(s) = 2\frac{\l}{\e}s < \frac{2s}{\e+\sqrt{\e^2-4\d s}}
=\frac{\e-\sqrt{\e^2-4\d s}}{2\d} = h'(s).
\end{equation}

Now we fix \(\l \in (0,\lubar)\), and for the sake of contradiction, we suppose that there exists \(a\in(-s,s)\) such that \(h(a)>p_\l(a)\).
The observations \eqref{h-p}, \eqref{hp-ns} and \eqref{hp-s} collectively lead to the following: there exist two points \(b\) and \(c\) with \(-s<b<c<s\) such that 
\begin{align*}
&h(b)=p_\l(b), &&h'(b)-p_\l'(b) \ge 0,
&&h(c)=p_\l(c), &&h'(c)-p_\l'(c) \le 0.
\end{align*}
%\begin{align*}
%&h(b)=p_\l(b), &&h'(b)-p_\l'(b) \ge 0\\
%&h(c)=p_\l(c), &&h'(c)-p_\l'(c) \le 0.
%\end{align*}
Then, we consider a linear function \(g_\l:[-s,s]\to\RR\) which is defined by
\[
g_\l(a) \coloneqq
\frac{\e}{\d} \Big(1-\frac{1}{2\l}\Big) - 2\frac{\l}{\e}a.
\]
Since \(h(b)=p(b)\) and \(h(c)=p(c)\), using \eqref{h'}, we obtain
\begin{align*}
g_\l(b)&=
\frac{1}{\d}\Big(\e-\frac{\e}{2\l}\Big)
-2\frac{\l}{\e}b
=\frac{1}{\d}\bigg(\e-\frac{(b-s)(b+s)}{2p_\l(b)}\bigg)
-2\frac{\l}{\e}b
=h'(b)-p_\l'(b) \ge 0,\\
g_\l(c)&=
\frac{1}{\d}\Big(\e-\frac{\e}{2\l}\Big)
-2\frac{\l}{\e}c
=\frac{1}{\d}\bigg(\e-\frac{(c-s)(c+s)}{2p_\l(c)}\bigg)
-2\frac{\l}{\e}c
=h'(c)-p_\l'(c) \le 0.
\end{align*}
Moreover, since \(\frac{\l^2}{1-2\l}<1\le\frac{\e^2}{4\d s}\), we have
\[
g_\l(-s)
=\frac{\e}{\d} \Big(1-\frac{1}{2\l}\Big) + 2\frac{\l}{\e}s < 0.
\]
This contradicts to the fact that the function \(g_\l\) is decreasing.
Thus, for any \(\l\in(0,\lubar)\), we obtain
\[
h(a) \le p_\l(a), \qquad \forall a\in(-s,s).
\]
Therefore, we conclude that \(h(a) \le p_{\lubar}(a)\) for all \(a\in(-s,s)\).

\vspace{2mm}
\case{2} We consider the second case when \(\l>\lbar = 1\).\\
This can be handled similarly.
Note that
\[
\l > \lbar > \frac{1}{2} > \frac{\e}{\e+\sqrt{\e^2+4\d s}},
\]
from which the following can be obtained:
\begin{equation} \label{hp-ns0}
- p_\l'(-s) = 2\frac{\l}{\e}s > \frac{2s}{\e+\sqrt{\e^2+4\d s}}
= -\frac{\e-\sqrt{\e^2+4\d s}}{2\d} = -h'(-s).
\end{equation}
Then, using
\[
\l > \lbar = 1
\ge \frac{\e}{\e+\sqrt{\e^2-4\d s}},
\]
we obtain
\begin{equation} \label{hp-s0}
p_\l'(s) = 2\frac{\l}{\e}s > \frac{2s}{\e+\sqrt{\e^2-4\d s}}
=\frac{\e-\sqrt{\e^2-4\d s}}{2\d} = h'(s).
\end{equation}

Then we fix \(\l \in (\lbar,\infty)\), and we suppose that there exists \(a\in(-s,s)\) such that \(h(a)<p_\l(a)\).
Thanks to \eqref{h-p}, \eqref{hp-ns0} and \eqref{hp-s0}, we choose two points \(b\) and \(c\) with \(-s<b<c<s\) such that 
\begin{align*}
&h(b)=p_\l(b), &&h'(b)-p_\l'(b) \le 0,
&&h(c)=p_\l(c), &&h'(c)-p_\l'(c) \ge 0.
\end{align*}
%\begin{align*}
%&h(b)=p_\l(b), &&h'(b)-p_\l'(b) \le 0\\
%&h(c)=p_\l(c), &&h'(c)-p_\l'(c) \ge 0.
%\end{align*}
In the same way, we show that the linear function \(g_\l\) verifies
\[
g_\l(-s) > 0, \qquad
g_\l(b) \le 0, \qquad
g_\l(c) \ge 0,
\]
which contradicts to the linearity of the function \(g\).
Thus, for any \(\l\in(\lbar,\infty)\), we have
\[
h(a) \ge p_\l(a), \qquad \forall a\in(-s,s).
\]
Thus, we find that \(h(a) \ge p_{\lbar}(a)\) for all \(a\in(-s,s)\).
This completes the proof of \eqref{key:ineq2} and \eqref{key:ineq}.
\end{pf}

\section{Proof of Theorem \ref{thm:main}: Contraction Property}
\setcounter{equation}{0}
The key ingredient in the proof is the Poincar\'e-type inequality, Lemma \ref{KV_lemma}.
To facilitate its application, we transform the spatial variable \(x\) into a bounded variable \(z\).
To this end, we introduce the following variable: 
\begin{equation} \label{ztu}
z(x) \coloneqq \util(x), \qquad
dz = \util' dx,
\end{equation}
which is well defined due to the monotonicity of the shock profiles.
This change of variable carries the unbounded domain \(\RR\) to the interval \((-s,s)\).
Note that Lemma \ref{lem:key} will serve as a Jacobian estimate of the change of variable \eqref{ztu}.

\vspace{2mm}
We now recall the definition of the shift, and Lemma \ref{lem:RHS} implies that 
\[
\frac{d}{dt}\frac{1}{2}\int_\RR (u^X-\util)^2 dx
= -\frac{1}{4s} \Big(\int_\RR (u^X-\util)\util'dx\Big)^2
-\frac{1}{2}\int_\RR (u^X-\util)^2 \util' dx
-\e\int_\RR ((u^X-\util)_x)^2 dx.
\]
Note that while the first and third term on the right-hand side are non-positive, the second term is non-negative.
Then, we introduce \(w \coloneqq (u^X-\util) \circ z^{-1}\) and rewrite the first and second term in terms of \(z\) and \(w\) as follows:
\[
-\frac{1}{4s} \Big(\int_\RR (u^X-\util)\util'dx\Big)^2
=-\frac{1}{4s} \Big(\int_{-s}^s w dz\Big)^2, \qquad
-\frac{1}{2}\int_\RR (u^X-\util)^2 \util' dx
= \frac{1}{2} \int_{-s}^s w^2 dz.
\]
Moreover, for the third term, we apply Lemma \ref{lem:key} to find that 
\[
-\e\int_\RR ((u^X-\util)_x)^2 dx
=\e \int_{-s}^s (w_z)^2 \Big(\frac{dz}{dx}\Big) dz
=- \int_{-s}^s (-\e\util') (w_z)^2 dz
\le -\lubar \int_{-s}^s (s-\util)(\util+s) (w_z)^2 dz.
\]
Thus, using Lemma \ref{lem:RHS}, we conclude that
\begin{align*}
\frac{d}{dt}\frac{1}{2}\int_\RR (u^X-\util)^2 dx
&\le -\frac{1}{4s} \Big(\int_{-s}^s w dz\Big)^2
+\frac{1}{2} \int_{-s}^s w^2 dz
-\lubar \int_{-s}^s (s-\util)(\util+s) (w_z)^2 dz\\
&\le
- \Big(\lubar-\frac{1}{4}\Big)\int_{-s}^s (s-\util)(\util+s) (w_z)^2 dz \le 0.
\end{align*}
This immediately implies \eqref{L2-stable}.

\vspace{2mm}
It only remains to prove \eqref{timeasy} and \eqref{shiftasy}.
The argument deriving \eqref{timeasy} and \eqref{shiftasy} from \eqref{L2-stable} was introduced in previous works \cite{Hur,CEKS} and relies on the estimate provided by \eqref{L2-stable}, without using any properties of the shock.
Hence, once \eqref{L2-stable} has been shown, \eqref{timeasy} and \eqref{shiftasy} follow immediately.
We therefore omit the details.
This ends the proof of Theorem \ref{thm:main}. \qed

\subsection{Remark on Oscillatory Shocks}
In this subsection, we briefly explain why the oscillatory shock case is much more difficult than
the monotone shock case and how to overcome these difficulties in the companion paper \cite{CEKS}.
To begin with, since the derivative of the shock does not have a fixed sign, 
%the integrand of the following term in \eqref{RHS} does not have a definite sign either: 
%\[
%-\frac{1}{2}\int_\RR (u^X-\util)^2 \util' dx.
%\]
%On intervals where the shock is increasing, the above term gives a non-positive contribution to the right-hand side of \eqref{RHS}, while on the shock decreasing intervals, it gives a non-negative contribution.
%Thus, we can exploit this term on shock increasing intervals.
the variable introduced in \eqref{ztu} is not globally defined, which complicates the analysis.

\vspace{2mm}
The natural approach is to apply the Poincar\'e-type inequality in Lemma \ref{KV_lemma} on each monotone interval, in particular on decreasing intervals.
Note that the Poincar\'e-type inequality requires two terms (on the right-hand side in \eqref{KV}).
As seen in the proof for the monotone case, the first term comes from the dissipation term and the second term arises from the shift term.
To convert the dissipation term into the first term, the following inequality (the first inequality in \eqref{key}) is required: 
\[
\l_i (u_i-\util(x))(\util(x)-u_{i-1})
\le -\e \util'(x), \qquad \forall x\in\RR,
\]
where \(u_i\) and \(u_{i-1}\) are neighboring local maximum and local minimum, respectively.
Hence, such inequalities need to be established on all monotone intervals, and it turned out that the coefficient sequence \(\{\l_i\}\) grows exponentially as we move towards the left (see \cite[Proposition 4.1]{CEKS}).

\vspace{2mm}
Then we still need another term, i.e., the square of the average.
This is the most challenging part of \cite{CEKS}, as the shift part only yields the following `non-localized' squared average term:
\begin{equation} \label{non-l}
-\frac{1}{4s} \Big(\int_\RR (u^X-\util)\util'dx\Big)^2.
\end{equation}
This was sufficient in the monotone shock case, where the variable in \eqref{ztu} could be defined globally.
In the oscillatory shock case, however, we need to use the Poincar\'e-type inequality on each monotone interval; that is, the squared average is required on each monotone interval.
The authors developed an inductive argument to get the squared average term \cite{CEKS}.
The key observation in that argument is as follows: since the shock is oscillatory, we choose \(a,b\in\RR\) such that \(\util(a)=\util(b) \eqqcolon u_*\).
Then,
\[
\Big(\int_a^b (u^X-\util)\util' dx\Big)^2
=\Big(\int_a^b (u^X-\util)_x (\util-u_*) dx\Big)^2
\le \Big(\int_a^b \big((u^X-\util)_x\big)^2 dx\Big)\Big(\int_a^b (\util-u_*)^2 dx\Big).
\]
This shows that the dissipation term yields the squared average term on \((a,b)\).
Then, if the squared average is available on the increasing interval \(\subset (a,b)\), the squared average on the (complement) decreasing interval \(\subset (a,b)\) can be obtained.
Thus, the non-localized squared average term \eqref{non-l} is first used to deal with the rightmost decreasing interval, leaving an extra squared average term.
This extra term can be then transferred inductively to the left, using the above observation, which allows us to obtain the required squared average term on each monotone interval.

\vspace{2mm}
In the proof, we need \(L^2\) estimates of the form \(\int_a^b (\util-u_*)^2 dx\), since a smaller value of this yields a larger amount of squared average term on \((a,b)\).
These estimates allow us to obtain the required amount.
To this end, detailed structural properties of the oscillatory shocks were established in \cite{CEKS}.

\section{Proof of Theorem \ref{thm:exp}: Exponential Decay}
\setcounter{equation}{0}
Finally, we establish the exponential decay properties of the monotone shock profiles.
The proof is based on a standard Gr\"onwall argument (see, for instance, \cite{EE-CMS,EEKO-JDE,acon_i_2}) together with the estimates on the derivatives of the shocks established in Lemma \ref{lem:key}.

\vspace{2mm}
We assume that \(\util(0)=0 \,(=\frac{u_-+u_+}{2})\) without loss of generality and recall \eqref{key:ineq} in Lemma \ref{lem:key}:
\begin{equation} \label{key}
\lubar (u_--\util(x))(\util(x)-u_+)
\le -\e \util'(x)
\le \lbar (u_--\util(x))(\util(x)-u_+), \qquad \forall x\in\RR,
\end{equation}
where \(u_-=s\) and \(u_+=-s\).
Note that \(\util'(x)<0\) for all \(x\in\RR\).
Then, we observe
\[
2s \ge \util(x)-u_+ \ge \util(0)-u_+ = s, \qquad \forall x\le0.
\]
This together with \eqref{key} implies that
\begin{equation} \label{keyleft}
\lubar s (u_--\util(x))
\le -\e (u_--\util(x))'
\le 2\lbar s (u_--\util(x)), \qquad \forall x\le0.
\end{equation}
Then, Gr\"onwall's lemma shows that for all \(x\le0\),
\begin{equation} \label{expleft}
s e^{-2\lbar s\frac{\abs{x}}{\e}}
=(u_--\util(0)) e^{2\lbar s\frac{x}{\e}}
\le u_--\util(x)
\le (u_--\util(0)) e^{\lubar s\frac{x}{\e}}
= s e^{-\lubar s\frac{\abs{x}}{\e}}.
\end{equation}
On the other hand, for the right-half line, we observe
\[
2s \ge u_--\util(x) \ge u_--\util(0) = s, \qquad \forall x\ge0.
\]
Combined with \eqref{key}, we find that 
\begin{equation} \label{keyright}
\lubar s(\util(x)-u_+)
\le -\e (\util'(x)-u_+)
\le 2\lbar s(\util(x)-u_+), \qquad \forall x\ge0.
\end{equation}
We, again, apply Gr\"onwall's lemma to obtain that for all \(x\ge0\), 
\begin{equation} \label{expright}
s e^{-2\lbar s\frac{\abs{x}}{\e}}
=(u_--\util(0)) e^{-2\lbar s\frac{x}{\e}}
\le \util(x)-u_+
\le (u_--\util(0)) e^{-\lubar s\frac{x}{\e}}
= s e^{-\lubar s\frac{\abs{x}}{\e}}.
\end{equation}
Lastly, plugging \eqref{expleft} and \eqref{expright} into \eqref{keyleft} and \eqref{keyright}, we obtain 
\[
\lubar s^2 e^{-2\lbar s \frac{\abs{x}}{\e}}
\le -\e \util'(x)
\le 2\lbar s^2 e^{-\lubar s \frac{\abs{x}}{\e}}
\]
which completes the proof of Theorem \ref{thm:exp}. \qed
\vspace{2mm}
\noindent\textbf{Acknowledgement.}
The first author is partially supported by National Science Foundation (DMS-2306258).
The second and third authors were supported by Samsung Science and Technology Foundation under Project Number SSTF-BA2102-01.
The fourth author is partially supported by National Science Foundation (DMS-2206218).
%The first and fourth authors are partially supported by a SQuaRE at the American Institute of Mathematics.

%\vspace{2mm}
%\noindent\textbf{Declaration of competing interest.}
%The authors declared that they have no conflict of interest to this work.

%\vspace{2mm}
%\noindent\textbf{Data availability statement.}
%We do not analyze or generate any datasets, because our work proceeds within a theoretical and mathematical approach. 

\let\OLDthebibliography\thebibliography
\renewcommand{\thebibliography}[1]{
  \OLDthebibliography{#1}
  \setlength{\itemsep}{0.2pt}
  \setlength{\parskip}{0.2pt}
}

\bibliographystyle{plain}
\bibliography{reference}

\end{document}